\documentclass[12pt]{extarticle}
\usepackage{extsizes}
\usepackage[reqno]{amsmath} % equation numbers on the right
\usepackage{amsmath}
\usepackage{amssymb}
\usepackage{amsthm}
\usepackage{amscd}
\usepackage{amsfonts}
\usepackage{graphicx}
\usepackage{dsfont}
\usepackage{fancyhdr}
\usepackage{enumerate}
\usepackage{youngtab}
\usepackage[utf8]{inputenc}
\usepackage{comment}

\usepackage{subfigure}
\usepackage[margin=1.5in]{geometry}
\usepackage{graphicx}
\usepackage{algorithm}
\usepackage{algpseudocode}
\usepackage{color}
\usepackage[hidelinks]{hyperref}
\usepackage{notation}

\theoremstyle{theorem}

\newtheorem{corollary}{Corollary}

\newtheorem{lemma}[corollary]{Lemma}
\newtheorem{lemma*}[lem6]{Lemma}

\newtheorem{theorem}[corollary]{Theorem}

\begin{document}

\AtEndDocument{%
  \par
  \medskip
  \begin{tabular}{@{}l@{}}%
    \textsc{Gabriel Coutinho}\\
    \textsc{Dept. of Computer Science} \\ 
    \textsc{Universidade Federal de Minas Gerais, Brazil} \\
    \textit{E-mail address}: \texttt{gabriel@dcc.ufmg.br} \\ \ \\
    \textsc{Thomás Jung Spier} \\
    \textsc{Dept. of Computer Science} \\ 
    \textsc{Universidade Federal de Minas Gerais, Brazil} \\
    \textit{E-mail address}: \texttt{thomasjung@dcc.ufmg.br}
  \end{tabular}}

\title{Sums of squares of eigenvalues and the vector chromatic number}
\author{Gabriel Coutinho\footnote{gabriel@dcc.ufmg.br} \and Thomás Jung Spier\footnote{thomasjung@dcc.ufmg.br}}
\date{\today}
\maketitle
\author
\vspace{-0.8cm}

\begin{abstract} 
   In this short paper we prove that the sum of the squares of negative (or positive) eigenvalues of the adjacency matrix of a graph is lower bounded by the sum of the degrees divided by the vector chromatic number, resolving a conjecture by Wocjan, Elphick and Anekstein (2018).
\end{abstract}

\begin{center}
\end{center}

%%%%%%%%%%%%%%%%%%%%%%%%%%%%%%%%%%%%%%%%%%%%%%%%%%%%%%%%%%%%%%%%%%%%%%%%%%%%%%%%

%\subsubsection*{\centering Statement of the result}\label{intro}
\ \\ \
Let $G$ be a graph, and $A$ be its adjacency matrix. In this paper we are concerned with the sums of squares of positive and negative eigenvalues, thus it will be convenient to express the spectral decomposition of $A$ as
\[
A = \sum_{\theta > 0} \theta E_\theta + \sum_{\lambda< 0} \lambda E_\lambda.
\]
Denote
\[
A^+ = \sum_{\theta > 0} \theta E_\theta \quad \text{and} \quad A^- = - \sum_{\lambda< 0} \lambda E_\lambda,
\]
thus $A = A^+ - A^-$, and both $A^+$ and $A^-$ are positive semidefinite.

Let $s^+$ and $s^-$ denote the sum of the squares of the positive and negative eigenvalues of $A$, respectively, thus
\[
s^+ = \lVert A^+ \rVert^2 \quad \text{and} \quad s^- = \lVert A^- \rVert^2.
\]

Denote the adjacency matrix of the complement of $G$ by $\overline{A}$. The vector chromatic number of $G$, denoted by $\chi_{\mathrm{vec}}(G)$, is given by the semidefinite program (sdp) below (see \cite[p. 428, Eq. (23)]{mceliece1977new}):

\begin{align} \max \quad & \langle J, Z\rangle \label{sdp1}\\
\textrm{subject to} \quad & \langle I, Z\rangle = 1 \nonumber\\
& Z \circ \overline{A} = 0 \nonumber\\
& Z\geq 0, \, Z\succeq 0. \nonumber
\end{align}

We prove the following result, answering a conjecture raised by Wocjan, Elphick and Anekstein \cite{wocjan2018more} and strengthening results due to Ando and Lin \cite{ando2015proof} and Guo and Spiro \cite{guo2022new}.

\begin{theorem}\label{thm:main_result} For every graph $G$ with $m$ edges, we have
\[
\min\{s^+,s^-\} \geq \frac{2m}{\chi_{\mathrm{vec}}(G)}.
\]
\end{theorem}

%%%%%%%%%%%%%%%%%%%%%%%%%%%%%%%%%%%%%%%%%%%%%%%%%%%%%%%%%%%%%%%%%%%%%%%%%%%%%%

Our proof below follows certain steps from Guo and Spiro \cite{guo2022new}, while avoiding the need to deal with homomorphisms to edge transitive graphs by exploiting an unexpected\footnote{To us.} observation about the sdp formulation of $\chi_{\mathrm{vec}}(G)$.

\subsubsection*{An alternative SDP formulation for $\chi_{\mathrm{vec}}(G)$}

As before, let $A$ and $\overline{A}$ denote the adjacency matrices of $G$ and its complement. Let $J$ be the all $1$s matrix, that is, $J = I + A + \overline{A}$.

\begin{lemma} \label{lem:chivec}
    Let $Z$ be a square matrix indexed by the vertices of a graph. Assume $Z \geq 0$ and $Z \succeq 0$. Then
    \[
    \langle J,Z \rangle \leq \chi_{\mathrm{vec}}(G) \langle J - A , Z \rangle.
    \]
\end{lemma}
\begin{proof}
First notice that
\[\langle J, Z\rangle\leq \chi_{\mathrm{vec}}(G) \langle J-A_G, Z\rangle\iff \langle J, Z\rangle\leq \chi_{\mathrm{vec}}(G) \langle I + \overline{A}, Z\rangle.\]

We claim that the inequality above holds for every $Z\geq 0$, $Z\succeq 0$. This is obtained by showing the stronger claim that $\chi_{\mathrm{vec}}(G)$ is equal to the following semidefinite program:

\begin{align} \max \quad & \langle J, Z\rangle \label{sdp2}\\
\textrm{subject to} \quad & \langle I+ \overline{A}, Z \rangle = 1 \nonumber \\
& Z\geq 0, \, Z\succeq 0. \nonumber
\end{align}
Every feasible solution of \eqref{sdp1} is clearly a feasible solution for \eqref{sdp2} with the same objective value. Thus $\chi_{\mathrm{vec}}(G)$ is smaller than or equal to the optimum of \eqref{sdp2}.

On the other hand, if $Z_0$ is an optimum solution for \eqref{sdp2}, let $S$ denote the set of edges $uv$ in $E(\overline{G})$ for which $(Z_0)_{uv} > 0$, and $e_u$ the characteristic vector of vertex $u$. Then
\[
\tilde{Z_0}:=Z_0 + \displaystyle\sum_{uv\in S}(Z_0)_{uv}(e_u-e_v)(e_u-e_v)^T
\]
is so that
\[
\tilde{Z_0} \geq 0, \ \tilde{Z_0} \succeq 0, \  \tilde{Z_0} \circ \overline{A} = 0, \ \langle \tilde{Z_0},I\rangle =  \langle Z_0,I+\overline{A}\rangle = 1,\ \text{and} \ \langle \tilde{Z_0},J\rangle =  \langle Z_0,J\rangle,
\]
therefore $\tilde{Z_0}$ is a feasible solution for \eqref{sdp1} with the same objective value as $Z_0$ in \eqref{sdp2}.
\end{proof}

We do not expect semidefinite program \eqref{sdp2} for $\chi_{\mathrm{vec}}(G)$ to be unknown, but we could not find a reference for it.

\subsubsection*{Remaining of the proof}

The remaining of the proof largely follows Lemma 2.2 from Guo and Spiro \cite{guo2022new}, which in turn closely relates to the neat argument from Ando and Lin \cite{ando2015proof}. For ease of notation in our summations below, we consider $E(G)$ to be a set of ordered pairs $uv := (u,v)$ which contains all pairs $(u,v)$ for which $\{u,v\}$ is an edge of the graph, and analogously for when we write $uv \not\in E(G)$.

\begin{lemma}\label{lem:guo_spiro} Let $G$ be a graph, and let $X$ and $Y$ be real positive semidefinite matrices with rows and columns indexed by $V(G)$. If $XY=0$, $X_{uv} = Y_{uv}$ whenever $uv \not \in E(G)$, and 
\[\lVert X \rVert^2\leq (1+\mu) \cdot 
\sum_{uv \not \in E(G)}X_{uv}^2.\]
then,
\[\lVert X \rVert^2\leq \mu\cdot \lVert Y \rVert^2.\]
\end{lemma}
\begin{proof} Observe that, $XY=0$ implies that $\lVert X + Y \rVert^2=\lVert X - Y \rVert^2$, which is equivalent to
\[
\sum_{uv \in E(G)} (X_{uv} + Y_{uv})^2 + \sum_{uv \not\in E(G)}(X_{uv} + Y_{uv})^2 = \sum_{uv \in E(G)} (X_{uv} -Y_{uv})^2 + \sum_{uv \not\in E(G)}(X_{uv} - Y_{uv})^2,
\]
which by hypothesis on $X$ and $Y$ is equivalent to,
\[
\sum_{uv \in E(G)} (X_{uv} + Y_{uv})^2 + 4 \sum_{uv \not\in E(G)}X_{uv}^2 = \sum_{uv \in E(G)} (X_{uv} -Y_{uv})^2,
\]
therefore, introducing the placeholder $d$, we have
\[
d: = \sum_{uv \not\in E(G)}X_{uv}^2 = -\sum_{uv\in E(G)}X_{uv}Y_{uv}.
\]
By Cauchy-Schwarz applied to the sum on the right hand side, we obtain
\[
d^2 \leq \left(\sum_{uv \in E(G)} X_{uv}^2\right) \left( \sum_{uv \in E(G)} Y_{uv}^2\right) = (\lVert X \rVert^2 - d)(\lVert Y \rVert^2 -d ),
\]
which is equivalent to
\[\left(\lVert X\rVert^2 + \lVert Y\rVert^2\right)d\leq \lVert X\rVert^2\lVert Y\rVert^2.\]
By hypothesis we have, $\lVert X\rVert^2 \leq (1+\mu)d$, and thus, 
\[\lVert X\rVert^2+ \lVert Y\rVert^2\leq (1+\mu)\lVert Y\rVert^2 \implies \lVert X\rVert^2\leq \mu\lVert Y\rVert^2,\]
as we wanted.
\end{proof}

The proof of Theorem \ref{thm:main_result} is now immediate:

\begin{proof}[Proof of Theorem \ref{thm:main_result}]
    We apply Lemma~\ref{lem:guo_spiro} by making $X = A^+$ and $Y=A^-$, and $1+\mu = \chi_{\mathrm{vec}}(G)$. Clearly $XY = 0$, as the eigenspaces of $A^+$ and $A^-$ for nonzero eigenvalues are orthogonal, and $X_{uv} = Y_{uv}$ if $uv \not\in E(G)$, as $A = A^+ - A^-$. Moreover, by making $Z = X \circ X$, note that $Z \geq 0$, $Z \succeq 0$, and therefore, applying Lemma~\ref{lem:chivec}, we have
    \begin{align*}
        \lVert X \rVert^2 & = \text{sum of all entries of }X \circ X  \\ & = \langle J,Z\rangle \\ & \leq \chi_\mathrm{vec}(G) \langle J-A,Z\rangle  \\ & = \chi_\mathrm{vec}(G) \sum_{uv \not\in E(G)} X_{uv}^2,
    \end{align*}
    therefore, by Lemma~\ref{lem:guo_spiro}, $\lVert X \rVert ^2 \leq (\chi_\mathrm{vec}(G) - 1) \lVert Y \rVert^2$, that is, 
    \[
        s^+ \leq (\chi_\mathrm{vec}(G) - 1) s^-,
    \]
    and because $2m = s^+ + s^-$, the result follows for $s^-$. The proof for $s^+$ is obtained by interchanging the roles of $X$ and $Y$.
\end{proof}

%%%%%%%%%%%%%%%%%%%%%%%%%%%%%%%%%%%%%%%%%%%%%%%%%%%%%%%%%%%%%%%%%%%%%%%%%%%%%%%%

\subsubsection*{Acknowledgements}

We acknowledge fruitful conversations about this problem with Marcel K. de Carli Silva, Thiago Oliveira and Levent Tunçel.

%%%%%%%%%%%%%%%%%%%%%%%%%%%%%%%%%%%%%%%%%%%%%%%%%%%%%%%%%%%%%%%%%%%%%%%%%%%%%%%%
\bibliographystyle{plain}
\IfFileExists{references.bib}
{\bibliography{references.bib}}
{\bibliography{../references}}

%%%%%%%%%%%%%%%%%%%%%%%%%%%%%%%%%%%%%%%%%%%%%%%%%%%%%%%%%%%%%%%%%%%%%%%%%%%%%%%%
	
\end{document}